\documentclass[12pt]{amsart}
\usepackage{amssymb}
\usepackage{amsmath, amsthm}
\usepackage{amsfonts}
\usepackage{bbold}
\usepackage{color,soul}
\usepackage{hyperref}
\usepackage{graphicx}
\usepackage{float}
\usepackage[dvipsnames]{xcolor}
\usepackage[top=1in, bottom=1in, left=1in, right=1in]{geometry}

\theoremstyle{plain}  
\newtheorem{thm}{Theorem}[section] 
\newtheorem{lem}[thm]{Lemma} 
\newtheorem{cor}[thm]{Corollary}
\newtheorem{question}[thm]{Question}

\theoremstyle{definition} 

\newtheorem{conj}[thm]{Conjecture}

\newcommand{\eps}{\varepsilon}

\newcommand{\mc}[1]{\mathcal{#1}}
\newcommand{\hr}{\hat{r}}
\newcommand{\bound}{1.699} 
\newcommand{\dzero}{0.5406}
\newcommand{\ezero}{0.2703}
\title{Asymptotics of Ramsey numbers of double stars}
\author{
	Sergey Norin} \address{Department of Mathematics and Statistics, McGill University}
\author{Yue Ru Sun}  \address{Department of Mathematics and Statistics, McGill University}
\author{Yi Zhao} \address{
	Department of Mathematics and Statistics,
	Georgia State University }

\thanks{The first two authors are supported by an NSERC grant 418520. The third author is partially supported by NSF grant DMS-1400073.}

\begin{document}

\begin{abstract} A \emph{double star} $S(n,m)$ is the graph obtained by joining the center of a star with $n$ leaves to a center of a star with $m$ leaves by an edge. Let $r(S(n,m))$ denote the Ramsey number of the double star $S(n,m)$.
 In $1979$ Grossman, Harary and Klawe have shown that  $$r(S(n,m)) = \max\{n+2m+2,2n+2\}$$ for $3 \leq m \leq n\leq \sqrt{2}m$ and $3m \leq n$. They conjectured that equality holds for all $m,n \geq 3$. Using a flag algebra computation, we extend their result showing that $r(S(n,m))\leq n+~2m~+~2$ for $m \leq n \leq  \bound m$. On the other hand, we show that the conjecture fails for $\frac{7}{4}m~+~o(m)\leq n \leq \frac{105}{41}m-o(m)$. Our examples additionally give a negative answer to a question of Erd\H{o}s, Faudree, Rousseau and Schelp from $1982$.
\end{abstract}
\maketitle


\section{Introduction}
\label{intro}

The Ramsey number $r(G)$ of a graph $G$ is the least integer $N$ such that  any 2-coloring of edges of $K_N$ contains a monochromatic copy of $G$. The difficult problem of estimating Ramsey numbers of various graph families has attracted considerable attention since its introduction in the paper of Erd\H{o}s and Szekeres~\cite{ErdSze35}. See~\cite{CFSSurvey,Radsurvey} for recent surveys. 
Computing Ramsey numbers exactly appears to be very difficult in general, even for trees. However, determining the Ramsey numbers of stars is fairly straightforward. Harary~\cite{HarStars} has shown that
$$r(K_{1,n})=\begin{cases}
2n, &\mathrm{if}\; n\; \mathrm{is \; odd,} \\
2n-1, &\mathrm{if}\; n\; \mathrm{is \; even.}
\end{cases}$$ 
A natural direction in extending the above result is to consider double stars. 
A {\em double star} $S(n,m)$, where $n \geq m \geq 0$, is the graph consisting of the union of two stars, $K_{1,n}$ and $K_{1,m}$, and an edge called the {\em bridge}, 
joining the centers of these two stars.
Grossman, Harary and Klawe have established the following bounds on $r(S(n,m))$.
\begin{thm}[Grossman, Harary and Klawe~\cite{grossman}]\label{t:GHK}
	$$r(S(n,m)) = \begin{cases} 
	\max(2n+1, n+2m+2) &\mathrm{if}\; n\; \mathrm{ is \; odd\; and\:} m\leq 2, \\ 
	\max(2n+2, n+2m+2) &\mathrm{if}\; n\; \mathrm{ is \: even\: or\:} m\geq 3, \mathrm{and\:} n \leq \sqrt{2}m \mathrm{\:or\:} n \geq 3m,
	\end{cases}
	$$
\end{thm} 

They further conjectured that the restriction $n \leq \sqrt{2}m$ or $n \geq 3m$ is not necessary.

\begin{conj}[Grossman, Harary and Klawe~\cite{grossman}]\label{c:GHK}
	$r(S(n,m)) \leq 
	\max(2n+2, n+2m+2) 
	$ for all $n \geq m \geq 0$.
\end{conj} 

Our first result shows that the above conjecture is false for a wide range of values of $m$ and $n$.  
\begin{thm}\label{t:lower} For all $n \geq m \geq 0$,
	\begin{equation}\label{e:C5}
		r(S(n,m)) \geq \frac{5}{6}m+\frac{5}{3}n + o(m). 
	\end{equation}
Further, for $n \geq 2m$,
\begin{equation}\label{e:LK7}
	r(S(n,m)) \geq \frac{21}{23}m+ \frac{189}{115}n + o(m). 
\end{equation}
\end{thm} 

Note that the bounds in  Theorem~\ref{t:lower} imply that Conjecture~\ref{c:GHK} fails for  $$\frac{7}{4}m+o(m)~\leq~n~\leq~\frac{105}{41}m-o(m).$$ 
Theorem~\ref{t:lower} also provides a negative answer to a related more general question about Ramsey numbers of trees, which we now discuss. Let $T$ be a tree, and let $t_1$ and $t_2$, with $t_1 \leq t_2$, be the sizes of the color classes in the 2-coloring of $T$. Then $r(T) \geq 2t_1+t_2 - 1$. Indeed, one can color the edges of $K_{2t_1+t_2 - 2}$ in two colors so that the edges of the first color induce the complete bipartite graph $K_{t_1+t_2-1,t_1-1}$. Similarly, we have $r(T) \geq 2t_2 - 1$ by considering a 2-coloring of the edges of $K_{2t_{2}-2}$ with the first color inducing the complete bipartite graph $K_{t_2-1,t_2-1}$. Let $r_B(T) :=\max(2t_1+t_2 - 1,2t_2-1)$.
Burr~\cite{Burr74} conjectured that $r(T)=r_B(T)$ for every tree $T$. Grossman, Harary and Klawe~\cite{grossman} disproved Burr's conjecture, by showing that the Ramsey number of some double stars is larger than $r_B(T)$ by one. (See Theorem~\ref{t:GHK}.) They asked whether the difference $r(T)-r_B(T)$ can be arbitrarily large. 
Haxell, {\L}uczak and Tingley proved that Burr's conjecture is asymptotically true for trees with relatively small maximum degree.

\begin{thm}[Haxell, {\L}uczak and Tingley~\cite{haxell}]\label{thm:haxell}
For every $\eta >0$ there exists $\delta>0$ satisfying the following. If $T$ is a tree with maximum degree at most $\delta|V(T)|$ then $r(T) \leq (1 + \eta)r_B(T)$.
\end{thm} 

Finally, Erd\H{o}s, Faudree, Rousseau and Schelp~\cite{brooms} asked whether $r(T)=r_B(T)$ for trees $T$ with colors classes of sizes $|V(T)|/3$ and $2|V(T)|/3$. (Note that in the case the two quantities in the definition of $r_B(T)$ are equal, and that Theorem~\ref{t:GHK} does not cover this case for double star.)  
Theorem~\ref{t:lower} gives a negative answer to this question and to the above question of Grossman, Harary and Klawe by showing that $r(T)$ and $r_B(T)$ can differ substantially even for trees with colors classes of sizes $k$ and $2k$. Indeed, 
if $T=S(2k-1,k-1)$ we have $r_B(T)=4k-1$, but $r(T) \geq 4.2k-o(k)$ by (\ref{e:LK7}). 
 
Let us now return to upper bounds. Using Razborov's flag algebra method, we extend the results of Theorem~\ref{t:GHK} showing the following.   

\begin{thm} \label{mainthm}
\[r(S(n,m)) \leq n+2m+2\]
for $m \leq n \leq \bound (m+1)$.
\end{thm}


 
The paper is structured as follows. In Section~\ref{s:prelim} we show that the problem of finding $r(S(n,m))$ is essentially equivalent to the problem of characterizing the set of pairs $(\delta,\eta)$ such that there exists graph $G$ with minimum degree at least $\delta|V(G)|$, in which every two vertices have at least $\eta|V(G)|$ common non-neighbors.
(See Theorem~\ref{t:graph} for the precise statement.) In Section~\ref{s:valid} we analyze this set of pairs. In Section~\ref{s:back} we continue the discussion of Ramsey numbers of double stars and prove the consequences of the results of Section~\ref{s:prelim} in this context. In particular, we prove Theorems~\ref{t:lower} and~\ref{mainthm} and establish general asymptotic upper and lower bounds on Ramsey numbers of double stars, which differ by less than $2\%$. (See Theorem~\ref{t:bounds}). 

The paper uses standard graph theoretic notation. In particular, $N(v)$ denotes the neighborhood of a vertex $v$ in a graph $G$, when the graph is understood from context.

\section{From Ramsey numbers to degree conditions}\label{s:prelim}

In this section we prove  preliminary results which  allow us to break the symmetry between colors and replace the original Ramsey-theoretic problem by an equivalent problem with Tur\'{a}n-type flavor.  Let $(B,R)$ be a partition of the edges of $K_p$ into two color classes $B$ and $R$. For brevity we will say that $(B,R)$ is \emph{$(n,m)$-free} if $K_p$ contains no $S(n,m)$ with all the edges belonging to the same part of $(B,R)$. For $v \in [p]$ and $C \in \{B,R\}$, let $N_C(v)$ denote the set of vertices joined to $v$ by edges in $C$, and let $\deg_C(v)=|N_C(v)|$.

The first lemma that we need is due to Grossman, Harary and Klawe, but we include a proof for completeness. 

\begin{lem}[{\cite[Lemma 3.4]{grossman}}]\label{l:grossman}
Let $p \geq n+2m+2$, and let $(B,R)$  be an $(n,m)$-free partition of the edges of $K_p$. Then $\deg_C(v) \leq n+m$  for every $v \in [p]$ and $C \in  \{B,R\}$.
\end{lem}

\begin{proof}
Choose $v \in [p]$ and $C \in \{B,R\}$ such that $\deg_C(v)$ is maximum. Suppose for a contradiction that $\deg_C(v) \geq n+m+1$. We assume without loss of generality that $C=B$. If $\deg_B(u) \geq m+1$ for some $u \in N_B(v)$ then $K_p$ contains a double star $S(n,m)$ with edges in $B$ and bridge $uv$. Thus $\deg_R(u) \geq p - m-1 \geq m+n+1$ for every $u \in N_B(v)$. It follows that there exist $u,w \in N_B(v)$ such that $uw \in R$. In this case $(B,R)$ contains a double star $S(n,m)$ with edges in $R$ and the bridge $uw$, a contradiction.
\end{proof}

\begin{lem}\label{l:degrees}
	Let $p \geq n+2m+2$, and let $(B,R)$  be a partition of the edges of $K_p$. Then $(B,R)$ is $(n,m)$-free  if and only if for every $C \in \{B,R\}$ and every $uv \in C$
either	
\begin{equation}\label{e:union}
|N_C(u) \cup N_C(v)| \leq n+m+1
\end{equation}
or\begin{equation}\label{e:degree}
\deg_C(u) \leq n \mathrm{\ and \ } \deg_C(v) \leq n.
\end{equation}
\end{lem}

\begin{proof}
Clearly, if $uv \in C$ satisfies either (\ref{e:union}) or (\ref{e:degree}) then $uv$ is not a bridge of a monochromatic $S(n,m)$. Conversely, suppose that $uv \in C$ for some  $C \in \{B,R\}$ violates both  (\ref{e:union}) and (\ref{e:degree}). In particular, we may assume that $\deg_C(u) \geq n+1$. By Lemma~\ref{l:grossman}, we may further assume that $\deg_C(v) \geq p-n-m-1 \geq m+1$. Thus  $(B,R)$ contains a double star $S(n,m)$ with edges in $C$ and a bridge $uv$. (It can be constructed by first choosing $n$ neighbors of $u$ from $N_C(u) \cup N_C(v)$ which will serve as the leaves of $S(n,m)$ adjacent to $u$. We choose these neighbors outside of $N_C(v)$ whenever possible. Then at least $m$ elements of $N_C(v)$ will remain, and can serve as the leaves of $S(n,m)$ adjacent to $u$.)
\end{proof}

The next key lemma will allow us to break the symmetry between colors and replace the original Ramsey-theoretic problem by an equivalent problem with Tur\'{a}n-type flavor. 

\begin{lem}\label{l:uniquecolour}
Let $p \geq \max(2n+2,n+2m+2)$, and let $(B,R)$  be an $(n,m)$-free partition of the edges of $K_p$. Then there exists $C \in \{B,R\}$ such that $\deg_C(v) \leq n$ for all $v \in [p]$.
\end{lem}

\begin{proof} Suppose for a contradiction that there exists $v_1 \in [p]$ such that $\deg_B(v_1) \geq n+1$, and  $v_2 \in [p]$ such that $\deg_R(v_2) \geq n+1$. Then, as $p \geq 2n+2$, there exists a partition $(V_B,V_R)$ of $[p]$ such that  $\deg_B(v) \geq n+1$ for every $v \in V_B$, $\deg_R(v) \geq n+1$ for every $v \in V_R$, and $V_B,V_R \neq \emptyset.$ As $(B,R)$  is $(n,m)$-free it follows from Lemma~\ref{l:degrees} that 
	\begin{equation}\label{e:triangle1}
	|N_C(u) \cap N_C(v)| \geq p-n-m-1 \mathrm{\ for\ all\ } u \in V_B,v\in V_R, C \in \{B,R\} \mathrm{\ such\ that\ } uv \not \in C
	\end{equation}	
Let $b:[p]^2 \to \{0,1\}$ be the characteristic function of $B$, that is $b(uv)=1$ if and only if $\{u,v\} \in B$. Define $r:[p]^2 \to \{0,1\}$ analogously. Then (\ref{e:triangle1}) can be rewritten as
	\begin{equation}\label{e:triangle2}
     \sum_{w \in [p]}(b(uv)r(uw)r(vw)+ r(uv)b(uw)b(vw)) \geq  p-n-m-1 
	\end{equation}
for all $ u \in V_B,v\in V_R$. Summing (\ref{e:triangle2}) over all such pairs $u$ and $v$ we obtain
	\begin{align}\label{e:sum1}
	&\sum_{(u,w) \in V_B^2, v \in V_R}(b(uv)r(uw)r(vw)+ r(uv)b(uw)b(vw)) \notag \\&+
	\sum_{u \in V_B, (v,w) \in V_R^2}(b(uv)r(uw)r(vw)+ r(uv)b(uw)b(vw))\notag\\ &\geq  (p-n-m-1)|V_B||V_R| 
	\end{align}
On the other hand for every $v \in V_R$,
		\begin{equation}\label{e:sum2}
		\sum_{(u,w) \in V_B^2}(b(uv)r(uw)r(vw)+ r(uv)b(uw)b(vw)) =  
		|N_B(v) \cap V_B||N_R(v) \cap V_B| \leq \frac{1}{4}|V_B|^2		
		\end{equation}
Similarly for every $u \in V_B$,	
		\begin{equation}\label{e:sum3}
		\sum_{(v,w) \in V_R^2}(b(uv)r(uw)r(vw)+ r(uv)b(uw)b(vw))  \leq \frac{1}{4}|V_R|^2			
		\end{equation}
Thus  	\begin{align}\label{e:sum4}
&\sum_{(u,w) \in V_B^2, v \in V_R}(b(uv)r(uw)r(vw)+ r(uv)b(uw)b(vw)) \notag \\&+
\sum_{u \in V_B, (v,w) \in V_R^2}(b(uv)r(uw)r(vw)+ r(uv)b(uw)b(vw))\notag\\ 
&\leq  \frac{1}{4}(|V_R||V_B|^2 + |V_B||V_R|^2) 
\end{align}
Combining (\ref{e:sum1}) and  (\ref{e:sum4}) we obtain 
	\begin{equation}\label{e:conclusion1}
 p-n-m-1 \leq \frac{1}{4} (|V_R|+|V_B|)=\frac{p}{4}.
	\end{equation} 		
Inequality (\ref{e:conclusion1}) can be rewritten as $3p \leq 4m+4n+4$. However,
	$$3p \geq 2(n+2m+2)+(2n+2)=4m+4n+6,$$
implying the desired contradiction.	
\end{proof}

Lemma~\ref{l:uniquecolour} readily implies the following main result of this section.

\begin{thm}\label{t:graph}
Let $n \geq m \geq 0$ and  $p \geq \max(2n+2,n+2m+2)$ be integers. Then the following are equivalent
\begin{enumerate}
	\item[(i)] $p < r(S(n,m))$,
	\item[(ii)] there exists a graph $G$ with $|V(G)|=p$ such that $\deg(v) \geq p-n-1$ for every $v \in V(G)$ and $|N(v) \cup N(u)| \leq n+m+1$ for all $uv \in E(G)$.
\end{enumerate}
\end{thm}
\begin{proof}
	{\bf (i) $\Rightarrow$ (ii).} Let $(B,R)$ be an $(n,m)$-free partition of the edges of $K_p$. By Lemma~\ref{l:uniquecolour}, we assume without loss of generality that $\deg_R(v) \leq n$ for every $v \in [p]$, or equivalently  $\deg_B(v) \geq p-n-1 \geq n+1$. Let $G$ be the graph with $V(G)=[p]$ and $E(G)=B$. By Lemma~\ref{l:degrees} $|N(v) \cup N(u)| \leq n+m+1$ for all $uv \in E(G)$. Thus $G$ satisfies (ii).
	
	{\bf (ii) $\Rightarrow$ (i).} Let $(B,R)$ be a partition of the edges of the complete graph with the vertex set $V(G)$ such that $B=E(G)$. Then neither $B$ nor $R$ contains the edge set of a double star $S(n,m)$ by Lemma~\ref{l:degrees}. Thus  $p < r(S(n,m))$.
\end{proof}
\section{Valid points}\label{s:valid}

By Theorem~\ref{t:graph}, the function $r(S(n,m))$ is completely determined by the answer to the following question: For which triples $(p,d,s)$ does there exist a  graph $G$ with $|V(G)|=p$ such that $\deg(v) \geq d$ for every $v \in V(G)$ and $|N(v) \cup N(u)| \leq s$ for all $uv \in E(G)$?

We will be primarily interested in the asymptotic behavior of $r(S(n,m))$, and thus rather than answering the (likely very difficult) question above we analyze the following setting.

Given $0 \leq \delta,\eta \leq 1$ we say that a graph $G$ with $|V(G)|>1$ is a \emph{$(\delta,\eta)$-graph} if \begin{itemize}
	\item  $\deg(v)+1 \geq \delta|V(G)|$ for every $v \in V(G)$, and 
	\item $|N(v) \cup N(u)| \leq (1 -\eta)|V(G)|$ for all $uv \in E(G)$.
\end{itemize}
We say that $(\delta,\eta) \in [0,1]^2$ is \emph{directly valid} if there exists a $(\delta,\eta)$-graph, and we say that  $(\delta,\eta) \in [0,1]^2$ is \emph{valid} if it belongs to the closure of the set of directly valid points. Let $\mc{V} \subseteq [0,1]^2$ denote the set of valid points. Note, in particular, that if  $0 \leq x_2 \leq x_1$, $0 \leq y_2 \leq y_1$ and $(x_1,y_1) \in \mc{V}$ then $(x_2,y_2) \in \mc{V}$. 
Finally, a point  $(\delta,\eta) \in [0,1]^2$ is \emph{invalid} if it is not valid.

In this section we approximate the set of valid points. 

\begin{lem}\label{l:sparsify} For $n,\delta,\eta \geq 0$, let $G$ be a $(\delta n)$-regular graph with $|V(G)|=n$   such that $|N(v) \cup N(u)| \leq (1 -\eta)n$ for all $uv \in E(G)$. Then 
	$$\left(\frac{1}{n}+p\delta,1-\frac{2}{n} -2\left(\delta-\frac{1}{n}\right) p+(2\delta+\eta-1)p^2 \right) \in \mc{V}$$  
for every $p \in [0,1]$.	
\end{lem}	
\begin{proof}
 We will construct a ``random sparsified blow-up'' of $G$ as follows. Let $k$ be an integer, let $U$ be a set with $|U|=kn$, and let  $\phi: U \to V(G)$ be a map such that $|\phi^{-1}(v)|=k$ for every $v \in V(G)$. Let $G'$ be a random graph with $V(G')=U$ is constructed as follows. Let $uv \in E(G')$ if $\phi(u)=\phi(v)$, let $uv \not \in E(G')$ if $\phi(u) \neq \phi(v)$ and $\phi(u)\phi(v) \not \in E(G)$, and finally let $uv$ be an edge of $G$ with probability $p$ (independently for each edge) if $\phi(u)\phi(v) \in E(G)$. (It is natural to think of $G'$ as a graph obtained from $G$ by replacing every vertex by a clique of size $k$ and every edge by a random bipartite graph with density $p$.) 
 
We have almost surely $\deg(v) \geq (1+p\delta n)k -o(k)$ for each $v\in V(G')$. Furthermore, let $uv \in E(G')$ be such that $\phi(v) \neq \phi(u)$, and let $\eta'=|N(\phi(v)) \cap N(\phi(u))|/n$, then $2\delta-\eta' \leq 1-\eta$, and almost surely 
 \begin{align*}
 nk&-|N(v) \cup N(u)| \geq kn(1-2\delta+\eta' + 2(\delta-\eta'-1/n)(1-p)+\eta'(1-p)^2)+o(k) \\ 
 &=kn(1-2/n -2(\delta-1/n)p+\eta'p^2) \geq kn(1-2/n  -2(\delta-1/n)p+(2\delta+\eta-1)p^2)      
 \end{align*}
 Thus $G'$ is almost surely a $(1/n+p\delta-o(1), 1-2/n  -2(\delta-1/n)p+(2\delta+\eta-1)p^2 -o(1))$-graph.
 It follows that $(1/n+p\delta,1-2/n  -2(\delta-1/n) p+(2\delta+\eta-1)p^2) \in \mc{V}$.
\end{proof}	

\begin{cor}\label{c:valid} For every $p \in [0,1]$,
	$$\left(\frac{1+2p}{5}, \frac{3-2p}{5} \right),\left(\frac{1+10p}{21}, \frac{19-18p+5p^2}{21}\right) \in \mc{V}$$
\end{cor}
\begin{proof}
	We apply Lemma~\ref{l:sparsify} to the cycle of length five ($n=5$, $\delta = 2/5$, $\eta=1/5$), and the line graph of the complete graph on seven-vertices ($n=21$, $\delta = 10/21$, $\eta=6/21$), respectively.
\end{proof}	

We use Corollary~\ref{c:valid} to approximate $\mc{V}$ from below. Approximating $\mc{V}$ from above requires the use of flag algebras.

\newcommand{\numInvalidPairs}{9}
\newcommand{\numInvalidPairsPlus}{10}
\begin{table}[h]
	\centering
	\begin{tabular}{|l| c | c |}
		\hline
		$i$ &$\delta^*_i$ & $\eta^*_i$ \\ \hline
		%
		1&0.505&0.3164\\ \hline 
		2&0.510&0.3080\\ \hline 
		3&0.515&0.3011\\ \hline 
		4&0.520&0.2944\\ \hline 
		5&0.525&0.2883\\ \hline 
		6&0.530&0.2823\\ \hline 
		7&0.535&0.2766 \\ \hline
		8&0.540&0.2710\\ \hline 
		9&\dzero & \ezero \\ \hline
	\end{tabular}
		
	\caption{Invalid pairs $(\delta^*_i,\eta^*_i)$.}\label{tbl:invalid}
\end{table}

\begin{thm}\label{t:invalid}
The pairs $(\delta^*_i,\eta^*_i)$ for $1 \leq i \leq \numInvalidPairs$ given in Table~\ref{tbl:invalid} are invalid.  
\end{thm}

\begin{proof} The proof is computer-generated and consists of a flag algebra computation carried out in Flagmatic \cite{sliacan}. It is accomplished by executing the following script, which produces certificates of infeasibility  that can be found at \url{http://www.math.mcgill.ca/sun/double-star.html}.
	\begin{verbatim}
	from flagmatic.all import *
	p = GraphProblem(7, density=[("3:121323",1),("3:",1)], mode="optimization")
	p.add_assumption("1:",[("2:12(1)",1)],$\delta^*_i$)
	p.add_assumption("2:12",[("3:12(2)",1)],$\eta^*_i$)
	p.solve_sdp(show_output=True, solver="csdp")
	\end{verbatim}
As the usage of flag algebra computations to obtain similar bounds has become standard in the area in recent years (see a survey~\cite{RazSurvey}), and the method is described in great detail in a number of papers (e.g.~\cite{DHMNSCliques,HKNFlags,KeevashSurvey,PVCliques}), we avoid extensive discussion of the 
flag algebra setting. Essentially, nonexistence of  $(\delta^*_i-\eps,\eta^*_i-\eps)$-graphs  for some positive $\eps$ is proved by exhibiting a system of inequalities, involving homomorphism densities of seven vertex graphs, which has to hold in every $(\delta^*_i-\eps,\eta^*_i-\eps)$-graph, but which has no solutions.
\end{proof}

As we will see in Section~\ref{s:back} for the purposes of investigating Ramsey numbers we are primarily interested in the restriction of $\mc{V}$ to the region $[0.5,0.545] \times [0.265,0.32]$. The sets of points in this region, which are valid by Corollary~\ref{c:valid} or invalid by Theorem~\ref{t:invalid},  are shown in Figure~\ref{f:valid}.

\begin{figure}[H]
	\begin{center}
		\includegraphics[scale=0.5]{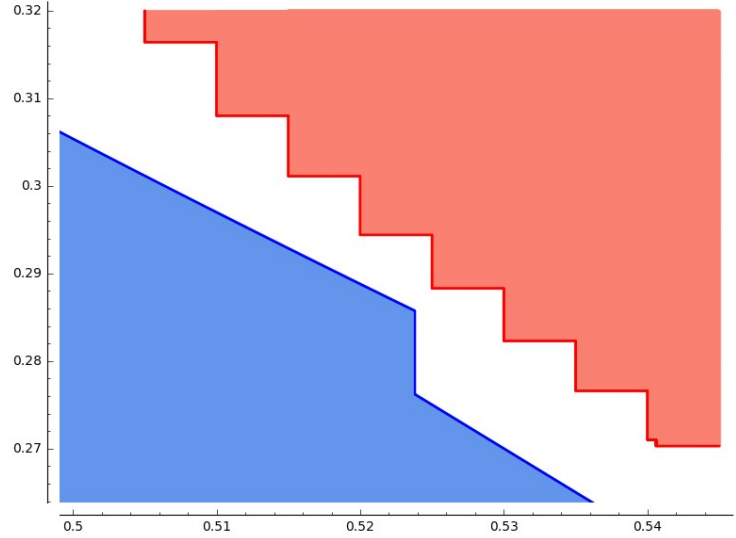}%
	\end{center}
	\caption{The restrictions of the valid {\color{blue}(blue)} and invalid {\color{red}(red)} point sets to the rectangle $[0.5,0.545] \times [0.265,0.32]$.}
	\label{f:valid}
\end{figure}

Finally, in addition to Theorem~\ref{t:invalid}, we will use the following result, which can be extracted from the proof of~\cite[Theorem 3.3]{grossman}. We include the proof for completeness.

\begin{thm}\label{t:invalid2}
	For every $\eps>0$ the pair $(1/2+\eps,1/3+\eps)$ is invalid. 
\end{thm}
\begin{proof}
	It suffices to show that, if $G$ is a graph with $|V(G)|=n$ such that $\deg(v) > n/2$ for every $v \in V(G)$, then there exists an edge $uv \in V(G)$ such that $|N(v) \cup N(u)| > 2n/ 3$. Suppose that no such edge exists. 
	
	For every pair of (not necessarily distinct) vertices $u,w \in V(G)$, there exists $v \in V(G)$ such that $uv, wv \in E(G)$. It follows that
	\begin{align}\label{e:nonadj}
	|N(u) &\cap N(w)| \geq |N(u) \cap N(w) \cap N(v)| \geq |N(u) \cap N(v)| +  |N(w) \cap N(v)| -|N(v)|\notag \\ 
	&= (|N(u)|+|N(v)| - |N(u) \cup N(v)| ) + (|N(w)|+|N(v)| - |N(w) \cup N(v)| ) -|N(v)| \notag\\ 
	&\geq |N(v)|+|N(u)|+|N(w)| -2\cdot \frac{2}{3}n  > \frac{n}{6}. 
	\end{align}
	Therefore,	
	\begin{align*}
	\frac{n^3}{4} &\geq \sum_{u \in V(G)}\deg(u)(n -\deg(u)) \\
	 &=  \sum_{u\in V(G), v\in N(u), w\not\in N(u)} 1 \\&= \sum_{\substack{(u,w) \in V(G)^2 \\ uw\not\in E(G)}} | N(u)\cap N(w)| + \sum_{\substack{(u,v) \in V(G)^2 \\ uv\in E(G)}} n - |N(u)\cap N(v)|\\ 
	  &\stackrel{(\ref{e:nonadj})}{>} \frac{n}{6}(n^2-2|E(G)|)+\frac{n}{3}\cdot 2|E(G)| \\
	&= \frac{n^3}{6}+ \frac{n}{3}|E(G)| \geq \frac{n^3}{4}, 
	\end{align*}
	a contradiction, as desired.
\end{proof}

\section{Back to Ramsey numbers}\label{s:back}

In this section we derive bounds on Ramsey numbers of double stars from the information on the set of valid points obtained in Section~\ref{s:valid}. In particular we prove Theorems~\ref{t:lower} and~\ref{mainthm}.

Our first lemma follows immediately from the definition of a directly valid point and Theorem~\ref{t:graph}.

\begin{lem}\label{l:upper}
	 Let $n \geq m \geq 0$ be integers, and let $p=r(S(n,m))-1$. If $p \geq \max(2n+2,n+2m+2)$ then $$\left(1 - \frac{n}{p},1 - \frac{n+m+1}{p}\right)$$
	is directly valid.
\end{lem}	

The next corollary is in turn a direct consequence of Lemma~\ref{l:upper}.
\begin{cor}\label{c:upper} Let $n \geq m \geq 0$ be integers, and let
 $(\delta, \eta)$ be an invalid point then 
	$$r(S(n,m)) \leq \max \left(2n+2,n+2m+2,\left\lceil\frac{n}{1-\delta}\right\rceil,\left\lceil\frac{n+m+1}{1 - \eta}\right\rceil\right).$$
\end{cor}

We are now ready to derive Theorem~\ref{mainthm}. 

\begin{proof}[Proof of Theorem~\ref{mainthm}.] By Theorem~\ref{t:invalid} the point $(\dzero,\ezero)$ is invalid. For $n \leq \bound (m+1)$ we have
	$$ \frac{n}{1 - \dzero} \leq n+2m+2\qquad  \mathrm{and} \qquad 
	\frac{n+m+1}{1 - \ezero} \leq n+2m+2.$$
Thus $r(S(n,m)) \leq n+2m+2$ by Corollary~\ref{c:upper}.
\end{proof}

Next we turn to asymptotic bound on $r(S(n,m))$. For $x \in [1,+\infty)$ define
\[\hr(x) := \lim_{\substack{n,m \to \infty \\n/m \to x}} \frac{r(S(n,m))}{m}. \]

The next theorem expresses $\hr(x)$ in terms of $\mc{V}$.

\begin{thm}\label{t:asymptotic} For every $x \geq 1$, let $$\hr'(x)=\max\left\{r \: : \: \left( 1-\frac{x}{r}, 1 -\frac{x+1}{r} \right) \in \mc{V} \right\}.$$ Then
	$$\hr(x) = \max(2x,x+2,\hr'(x)).$$
	In particular, the limit in the definition of $\hr(x)$ exists.
\end{thm}
\begin{proof} 
	It follows immediately from Corollary~\ref{c:upper} that  $$\limsup_{\substack{n,m \to \infty \\n/m \to x}} \frac{r(S(n,m))}{m} \leq  \max(2x,x+2,\hr'(x)).$$
	Since $r(S(n, m))\ge \max (2n+2, n+2m + 2)$ for $n \geq m \geq 3$, it remains to show that for all $x \geq 1$ and $\eps >0$ there exist $\gamma, N>0$ such that, if $m \geq N$ $n \geq (x -\gamma)m$ are integers, then $r(S(n,m)) \geq (\hr'(x)-\eps)m$. 
	Let $$r = \hr'(x), \; \delta= 1-\frac{x}{r}-\frac{\eps}{2r^2}, \; \mathrm{and} \; \eta= 1-\frac{x+1}{r}-\frac{\eps}{2r^2}.$$ The point $(\delta, \eta)$ is directly valid by definition of $\mc{V}$. Therefore there exists a graph $G$ with $|V(G)|>1$  such that  $\deg(v)+1 \geq \delta |V(G)|$ for every $v \in V(G)$ and $|N(v) \cup N(u)| \leq (1 -\eta)|V(G)|$ for all $uv \in E(G)$.
	Let $s=|V(G)|$, $\gamma=\frac{\eps^2}{4r^2}$, $N=\lceil 2s /\eps\rceil$. Let $m \geq N$, $n \geq (x -\gamma)m$ be integers, and   let $k= \lfloor (r - \eps/2)m/s \rfloor$.
	 Then 
	\begin{equation}\label{e:deltabound}
	(1 - \delta) ks \leq \left( \frac{x}{r}+\frac{\eps}{2r^2}\right) \left( r -\frac{\eps}{2}\right)m\leq \left(x-\frac{\eps^2}{4r^2} \right)m\leq n.
	\end{equation}
	Similarly,
	\begin{equation}\label{e:etabound}
		(1 - \eta) ks \leq \left( \frac{x+1}{r}+\frac{\eps}{2r^2}\right) \left( r -\frac{\eps}{2}\right)m\leq  \left(x+1-\frac{\eps^2}{4r^2}\right) m \leq m+n.
	\end{equation}
Let $G'$ be the graph with $|V(G)|=ks$ obtained by replacing every vertex of $G$ by a complete graph on $k$ vertices and replacing all the edges by complete bipartite graphs. By (\ref{e:deltabound}) and (\ref{e:etabound}), the graph $G'$  satisfies the conditions in Theorem~\ref{t:graph}(ii). Therefore $$r(S(n,m)) \geq ks =\lfloor (r -\eps/2)m/s \rfloor s \geq (r -\eps/2)m -s \geq (r -\eps)m$$
	by the choice of $N$, as desired.
\end{proof}

\begin{cor}\label{c:linlower} The following inequalities hold:
\begin{align}
	 &\hr(x) \geq \frac{5}{3}x+ \frac{5}{6} & \mathrm{for \ } 1 \leq x,  \label{e:lower1}\\
	 &\hr(x) \geq \frac{21}{10}x  &\mathrm{for \ } 1 \leq x \leq 2,\label{e:lower2}\\
	 &\hr(x) \geq \frac{189}{115}x+\frac{21}{23}  &\mathrm{for \ }  2 \leq x .\label{e:lower3}	  
	 \end{align}
\end{cor}
\begin{proof}
	For $x \geq 1$, let $p_1=(x+2)/(2x+1) \leq 1$, and let $r_1 = 5x/3+ 5/6$. Direct computations show that $1-x/r_1=(1+2p_1)/5$, and  $1-(x+1)/r_1=(3-2p_1)/5$. By Corollary~\ref{c:valid}, $((1+2p_1)/5,(3-2p_1)/5) \in \mc{V}$. Thus $\hr(x) \geq r_1$ by Theorem~\ref{t:asymptotic}, and (\ref{e:lower1}) holds.
	
	For the proof of  (\ref{e:lower2}), let $r_2 = 21x/10$. Then $1-x/r_2 =11/21$, and $1-(x+1)/r_2 \leq 6/21$ for  $1 \leq x \leq 2$. We have $(11/21,6/21) \in \mc{V}$ by  Corollary~\ref{c:valid} applied with $p=1$. Thus (\ref{e:lower2}) holds.
	
	Finally, for the proof of (\ref{e:lower3}), let $x \geq 2$, let $p_3=(20+13x)/(10+18x) \leq 1$, and let  $r_3 =189x/115 + 21/23$. Then 
	$1-x/r_3=(1+10p_3)/21$ and $$1- \frac{x+1}{r_3} = \frac{19-18p_3+5p_3^2}{21} - \frac{125}{84}\left(\frac{x-2}{5+9x}\right)^2 \leq  \frac{19-18p_3+5p_3^2}{21}.$$
	By Corollary~\ref{c:valid}, we have $((1+10p_3)/21,(19-18p_3+5p_3^2)/21) \in \mc{V}$, implying $\hr(x) \geq r_1$ by Theorem~\ref{t:asymptotic}.
\end{proof}	

Let us note that the lower bound in (\ref{e:lower3}) can be tightened to $$\hr(x) \geq \frac{7}{60}\left(5+4x+\sqrt{25+40x+106x^2} \right).$$
However, we chose to keep the bound linear, and hence hopefully more transparent.

Theorem~\ref{t:lower} follows immediately from Corollary~\ref{c:linlower}.

\begin{proof}[Proof of Theorem~\ref{t:lower}.] Note that $$r(S(n,m)) = \hr(\frac{n}{m})m+o(m).$$ Thus inequalities (\ref{e:C5}) and (\ref{e:LK7}) Theorem~\ref{t:lower} follow from inequalities (\ref{e:lower1}) and (\ref{e:lower3}) in Corollary~\ref{c:linlower}, respectively.
\end{proof}	

More generally, Corollary~\ref{c:linlower} implies the following piecewise linear lower bound on $\hr(x)$.
Let $$\hr_l(x)=
\begin{cases} x+2 &\mathrm{for \ } 1 \leq x \leq \frac{7}{4} \\
\frac{5}{3}x+ \frac{5}{6} & \mathrm{for \ } \frac{7}{4} \leq x \leq \frac{25}{13}, \\
\frac{21}{10}x  &\mathrm{for \ } \frac{25}{13} \leq x \leq 2,\\
\frac{189}{115}x+\frac{21}{23}  &\mathrm{for \ }  2 \leq x \leq \frac{105}{41}, \\
2x	 &\mathrm{for \ }  \frac{105}{41} \leq x. 
\end{cases}$$
Coming back to upper bounds, define  
$$u_{\delta,\eta}(x) = \max\left(x+2,2x, \frac{x}{1 -\delta},\frac{x+1}{1 -\eta} \right)$$
for $(\delta,\eta) \in [0,1]^2.$ It follows from Theorem~\ref{t:asymptotic} that $\hr(x) \leq u_{\delta,\eta}(x)$ for every point $(\delta,\eta)$ that lies in the closure of the set of invalid points. Accordingly we define $$\hr_u(x)=\min_{i=1}^{\numInvalidPairsPlus} u_{\delta^*_i,\eta^*_i}(x),$$
where the pairs $(\delta^*_i,\eta^*_i)$ for $1 \leq i \leq \numInvalidPairs$ are listed in Table~\ref{t:invalid}, and $(\delta^*_{\numInvalidPairsPlus},\eta^*_{\numInvalidPairsPlus})=(1/2,1/3)$. It follows from Theorems~\ref{t:invalid} and~\ref{t:invalid2} that $\hr_u(x)$ is an upper bound on $\hr(x)$. Next theorem collects all the asymptotic lower and upper bounds on Ramsey numbers of double stars established in this section.
\begin{thm}\label{t:bounds}$ \hr_l(x) \leq \hr(x) \leq \hr_u(x).$
\end{thm}

We present two figures which should be helpful in visualizing the bounds in Theorem~\ref{t:bounds}. For comparison, let us introduce an additional function $\hr^*_l(x)=\max(x+2,2x)$ equal to the value of $\hr(x)$ conjectured in~\cite{grossman}. Functions $\hr^*_l(x)$,$\hr_l(x)$ and $\hr_u(x)$ are plotted in Figure~\ref{f:bounds1}. The ratio $\hr_u(x)/\hr_l(x)$ is plotted in Figure~\ref{f:bounds2}.  
In particular, the bounds in Theorem~\ref{t:bounds} asymptotically predict the value of $r(S(n,m))$ with the error less that $1\%$. In comparison, as mentioned in the introduction, the value of $r(S(2m,m))$ conjectured in~\cite{grossman} is asymptotically smaller than  the lower bound provided in Theorem~\ref{t:lower} by 
$5 \%$.

\begin{figure}
	\begin{center}
		\includegraphics[scale=0.75]{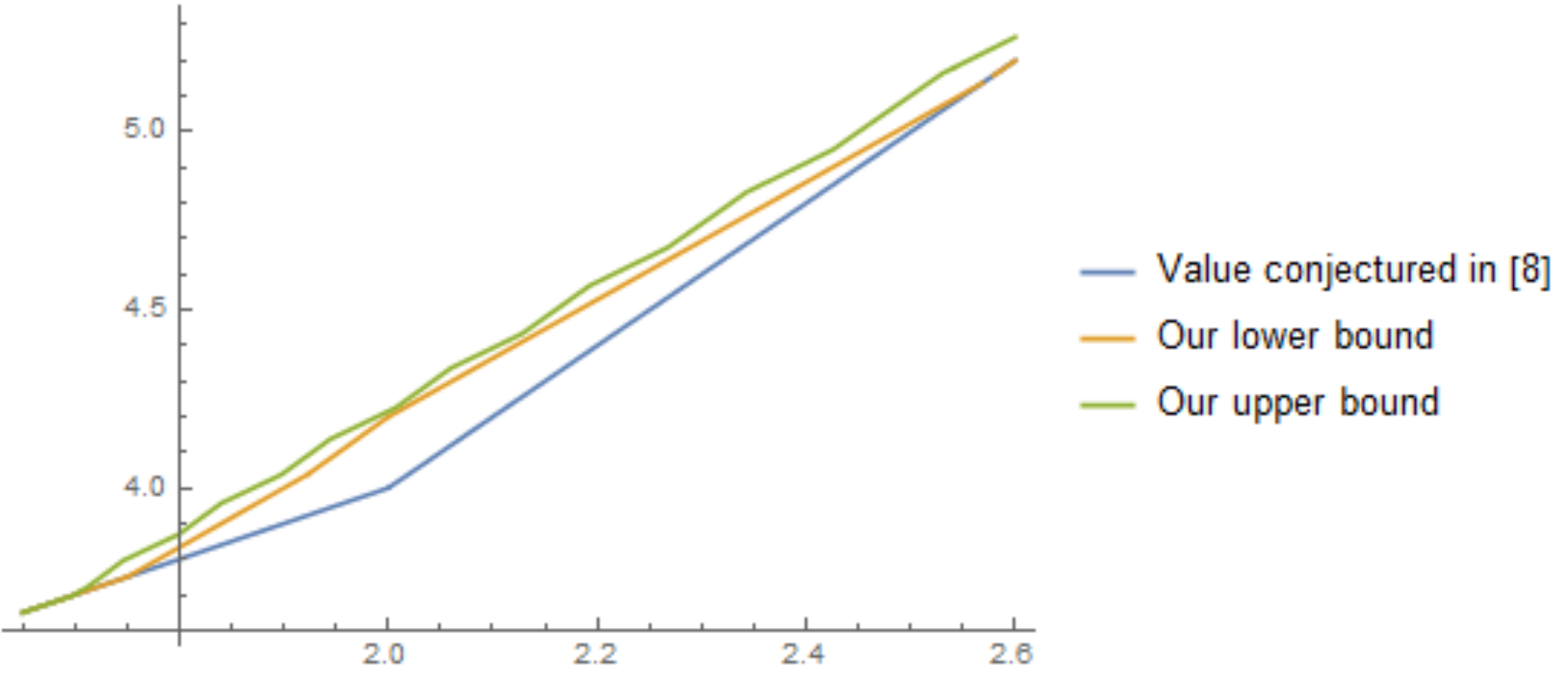}%
	\end{center}
	\caption{Functions $\hr^*_l(x)$,$\hr_l(x)$ and $\hr_u(x)$ on an interval $1.65 \leq x \leq 2.6$.}
	\label{f:bounds1}
\end{figure}

\begin{figure}
	\begin{center}
		\includegraphics[scale=0.75]{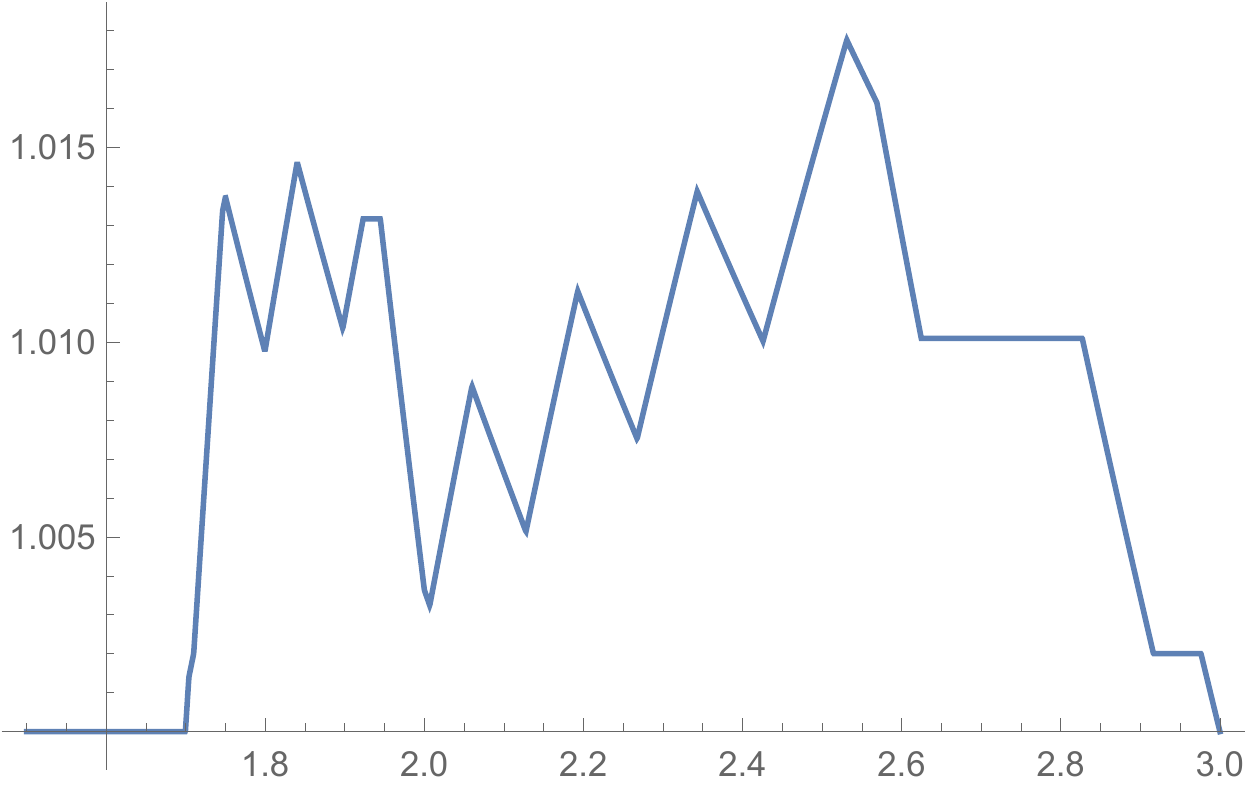}%
	\end{center}
	\caption{The ratio $\hr_u(x)/\hr_l(x)$ on an interval $1.5 \leq x \leq 3$.}
	\label{f:bounds2}
\end{figure}


\section{Concluding remarks} 

\subsection*{Asymptotic value of $r(S(n,m))$.}
The constructions we used to provide the new lower bounds on Ramsey numbers of double stars are not simple, and we do not attempt to conjecture their tightness. Understanding the asymptotic behavior of $r(S(2m,m))$ appears to be difficult already.

\begin{question}
	Is $r(S(2m,m)) =4.2m +o(m)$?
\end{question}	  

Perhaps, a combination of  flag algebra techniques with stability methods, along the lines of the arguments in~\cite{BHLP5Cycle,BHLPVYRainbow}, can be used to resolve the above question. More ambitiously, one can ask the following.

\begin{question}
Let $T$ be a tree on $n$ vertices. Suppose that color classes in the 2-coloring of $T$ have sizes $n/3$ and $2n/3$. Is $r(T) \leq 1.4n+o(n)$? 	
\end{question}	  

\subsection*{When is $r(S(n,m))=2n+2$?} In Theorem~\ref{mainthm} we were able to substantially extend the range of known values $(m,n)$ for which the equality $r(S(n,m))=n+2m+2$ holds. We were not similarly successful in reducing the lower bound on $n$ in Theorem~\ref{t:GHK} which guarantees $r(S(n,m))=2n+2$. By Theorem~\ref{t:graph}, finding the optimal bound is essentially equivalent to answering the following question. 

\begin{question} Find  the infimum $c_{\inf}$ of the set of real numbers $c$ for which  there exists a graph $G$ with $|V(G)|=n$ such that  
	\begin{itemize}
		\item $\deg(v) > n/2$ for every $v \in V(G)$, and
	    \item $|N(v) \cup N(u)| \leq cn$ for every $uv \in E(G)$.
	\end{itemize}
\end{question}	  
 
Theorem~\ref{t:invalid2} shows that $c_{\inf} \geq 2/3$. The sparsified blow-ups of the line graph of $K_7$, introduced in the proof of Lemma~\ref{l:sparsify}, show that $c_{\inf} \leq 389/560 \approx 0.694$.  We have convinced ourselves that $c_{\inf} > 2/3$,  but the proof is technical and does not provide a meaningful improvement of the lower bound on $c_{\inf}$.
 
\bibliographystyle{abbrv} 
\bibliography{ds-refs} 


\end{document}